\documentclass[11pt]{article}

\usepackage[a4paper,margin=25mm]{geometry}
\usepackage{authblk}
\usepackage[T1]{fontenc}
\usepackage[utf8]{inputenc}
\usepackage{lmodern}
\usepackage{amsmath,amssymb,amsthm,mathtools}
\usepackage{array,longtable,booktabs}
\newcolumntype{P}[1]{>{\raggedright\arraybackslash}p{#1}}
\usepackage{microtype}
\usepackage[dvipsnames]{xcolor}
\usepackage[colorlinks=true,linkcolor=MidnightBlue,citecolor=MidnightBlue,urlcolor=MidnightBlue]{hyperref}

\newtheorem{theorem}{Theorem}[section]

\newtheorem{lemma}[theorem]{Lemma}

\theoremstyle{remark}
\newtheorem{remark}[theorem]{Remark}

\newcommand{\F}{\mathbb F}
\newcommand{\Leib}{\operatorname{Leib}}
\newcommand{\Der}{\operatorname{Der}}

\newcommand{\Fx}{\F^{\times}}

\title{Derivation Algebras of Three-Dimensional Leibniz Algebras:\\A Complete Description}

\author[1]{Leonid A. Kurdachenko}
\author[1]{Oleksandr O. Pypka}
\author[2]{Mykola M. Semko}

\affil[1]{Oles Honchar Dnipro National University, Dnipro, Ukraine}
\affil[2]{State Tax University, Irpin, Ukraine}

\date{}

\begin{document}

\maketitle

\vspace{-3.5em}

\begin{center}
\small
E-mail addresses:\\[2pt]
\href{mailto:lkurdachenko@gmail.com}{lkurdachenko@gmail.com}
(L.A. Kurdachenko)\\
\href{mailto:sasha.pypka@gmail.com}{sasha.pypka@gmail.com}
(O.O. Pypka)\\
\href{mailto:dr.mykola.semko@gmail.com}{dr.mykola.semko@gmail.com}
(M.M. Semko)
\end{center}

\begin{abstract}
Derivation algebras are natural structural invariants of Leibniz algebras.  In a number of earlier papers, derivations were described for several classes of one-generated, nilpotent, non-nilpotent, and low-dimensional Leibniz algebras.  In the present paper we complete the description for three-dimensional non-Lie left Leibniz algebras over arbitrary fields.  We use a refined organization of the three-dimensional classification into sixteen isomorphism types, including parameterized families.  Previously known cases are incorporated by reference, with changes of basis recorded when necessary.  For the remaining cases we determine the derivations by direct coefficient calculations.  Special attention is paid to characteristic two, where several derivation algebras increase in dimension.  We also record simple decompositions of the resulting Lie algebras into natural ideals and subalgebras.  A final table gives the derivation algebra of every type in the classification.
\end{abstract}



\section{Introduction}

Let $L$ be an algebra over a field $\F$.  A linear map $d:L\to L$ is called a derivation if
\[
 d([x,y])=[d(x),y]+[x,d(y)]
 \qquad (x,y\in L).
\]
The set $\Der(L)$ of all derivations is a Lie algebra under the usual commutator
\[
 [d_1,d_2]=d_1d_2-d_2d_1.
\]
Thus an explicit description of $\Der(L)$ gives a natural invariant of the multiplication of $L$.

Derivations of Leibniz algebras have been studied in many different situations. In particular, derivations play a substantial role in the structural theory developed for Leibniz algebras; see, for example,~\cite{Monograph} and the references therein. Finite-dimensional one-generated Leibniz algebras were considered in~\cite{DerCyclic1,DerCyclic2}.  Several nilpotent and low-dimensional cases were treated in~\cite{DerNilpotent,DerLowDimsReports,DerLowDims,DerLowDimsUMJ}, while non-nilpotent cases were studied in~\cite{DerNonNilpotent1,DerSome,DerNonNilpotent3,DerNonNilpotent4}.  Endomorphisms and derivations of certain infinite-dimensional one-generated Leibniz algebras were considered in~\cite{DerInfinite}.  Mancini~\cite{ManciniBider} computed biderivations of right Leibniz algebras of dimension at most three over fields of characteristic different from two.  Since the first component of a biderivation is a derivation, those calculations also provide a useful independent comparison in characteristic different from two.

A detailed classification of three-dimensional non-Lie Leibniz algebras over arbitrary fields was considered in~\cite{Old2022,Rakhimov2018}.  A subsequent reconsideration of the case division, together with the identification of isomorphic presentations and a refinement of several parameter and characteristic restrictions, allows the resulting list to be organized more economically into sixteen types.  We use this refined organization throughout the present paper.

The available derivation results do not immediately give a complete arbitrary-field table for these sixteen representatives.  Some papers use different multiplication tables, some results are stated only under a characteristic restriction, and characteristic two produces additional derivations in several cases.  Moreover, for a few representatives the displayed formula in an earlier source does not give the full derivation space in the present normalization.  For this reason we compare the earlier results with the fixed sixteen-type list and give direct calculations whenever this is needed.

The main result is the following.

\begin{theorem}[Main theorem]\label{thm:main}
Let $L$ be a three-dimensional non-Lie left Leibniz algebra over an arbitrary field $\F$.  Then $\Der(L)$ is explicitly determined by the formulas proved or cited in the present paper and summarized in Table~\ref{tab:summary}.  In particular, all parameter families and all characteristic-two cases in the refined sixteen-type classification are covered.
\end{theorem}

\section{Preliminaries and the sixteen types}

Throughout the paper, $\F$ is an arbitrary field unless a characteristic restriction is stated explicitly.  All Leibniz algebras are \emph{left} Leibniz algebras, so
\begin{equation}\label{eq:Leibniz}
 [x,[y,z]]=[[x,y],z]+[y,[x,z]]
\end{equation}
for all $x,y,z\in L$.  The Leibniz kernel is
\[
 \Leib(L)=\langle [x,x]\mid x\in L\rangle_{\F}.
\]
All products not displayed in a multiplication table are zero.

We fix a basis $a_1,a_2,a_3$ of $L$.  Matrices act on column vectors: the $j$-th column of the matrix of $d$ consists of the coordinates of $d(a_j)$.  When one algebra $L$ is fixed, we put $D=\Der(L)$ and denote by $\Xi:D\to M_3(\F)$ the canonical monomorphism which assigns to every derivation its matrix in the basis $a_1,a_2,a_3$.  As in our earlier papers, we identify $D$ with $\Xi(D)$ when writing matrix formulas.

Whenever the internal structure of $D$ is described, the basis derivations are defined explicitly from the parameters in the displayed matrix.  We use $\rtimes$ only in the elementary sense of a semidirect sum: if $D=I\rtimes H$, then $I$ is an ideal, $H$ is a subalgebra, $D=I+H$, and $I\cap H=0$.

We use the following elementary observation repeatedly.

\begin{lemma}\label{lem:invariant}
Let $d\in\Der(L)$.  Then
\[
 d(\Leib(L))\subseteq\Leib(L),\qquad d([L,L])\subseteq[L,L].
\]
The left center, the right center, and the center of $L$ are also invariant under $d$.
\end{lemma}

\begin{proof}
First,
\[
 d([x,y])=[d(x),y]+[x,d(y)]\in[L,L],
\]
so $[L,L]$ is invariant.  For the Leibniz kernel, note that
\[
 d([x,x]) = [d(x),x]+[x,d(x)] = [d(x)+x,d(x)+x]-[d(x),d(x)]-[x,x].
\]
Each term in the last expression belongs to $\Leib(L)$, and hence
$d(\Leib(L))\subseteq\Leib(L)$.

If $z$ belongs to the left center, then for every $x\in L$,
\[
 0=d([z,x])=[d(z),x]+[z,d(x)]=[d(z),x],
\]
so $d(z)$ again belongs to the left center.  The right center is treated in the same way by applying $d$ to $[x,z]=0$.  Their intersection, the center of $L$, is therefore also invariant.
\end{proof}

The sixteen types used below are listed in Table~\ref{tab:status}.  The last two columns record the literature situation and the role of the present paper.  The purpose of the table is not to reproduce every result from the cited papers, but to show which cases can be used directly and which cases require a new calculation in the present basis.

\begingroup
\small
\setlength{\tabcolsep}{3pt}
\renewcommand{\arraystretch}{1.18}
\begin{longtable}{@{}P{0.09\textwidth}P{0.415\textwidth}P{0.245\textwidth}P{0.19\textwidth}@{}}
\caption{The refined sixteen-type list and the status of the derivation problem.}\label{tab:status}\\
\toprule
Type & Nonzero products and restrictions & Earlier derivation results & Status here\\
\midrule
\endfirsthead
\toprule
Type & Nonzero products and restrictions & Earlier derivation results & Status here\\
\midrule
\endhead
\midrule
\multicolumn{4}{r}{\emph{Continued on next page}}\\
\endfoot
\bottomrule
\endlastfoot
$L_1$ & $[a_1,a_1]=a_3$ & \cite{DerLowDimsUMJ,DerLowDimsReports} & previously known\\
$L_2$ & $[a_1,a_1]=[a_1,a_2]=a_3$ & \cite{DerNilpotent} & previously known\\
$L_3(\alpha)$ & $[a_1,a_1]=[a_2,a_1]=a_3$, $[a_1,a_2]=\alpha a_3$; $\alpha\ne0,-1$, $\alpha\sim\alpha^{-1}$ & no complete arbitrary-field description located & arbitrary fields completed here\\
$L_4$ & $[a_1,a_1]=[a_2,a_1]=a_3$, $[a_1,a_2]=-a_3$ & no complete arbitrary-field description located & characteristic $2$ and unified formula here\\
$L_5(\beta)$ & $[a_1,a_1]=a_3$, $[a_2,a_2]=\beta a_3$; $X^2+\beta$ has no root; $\beta'/\beta\in(\F^\times)^2$ gives the same isomorphism type & related calculation in \cite{DerLowDims} & full arbitrary-field calculation here\\
$L_6(\delta)$ & $[a_1,a_1]=[a_1,a_2]=a_3$, $[a_2,a_2]=\delta a_3$; $X^2+X+\delta$ has no root; in this normalization $\delta$ is an exact parameter & related calculation in \cite{DerLowDims} & full arbitrary-field calculation here\\
$L_7$ & $[a_1,a_1]=a_3$, $[a_1,a_2]=a_2$, $[a_2,a_1]=-a_2$ & \cite{DerNonNilpotent3} & previously known\\
$L_8(\lambda)$ & $\operatorname{char}\F=2$; $[a_1,a_1]=\lambda a_3$, $[a_1,a_2]=[a_2,a_1]=a_2$, $[a_2,a_2]=a_3$; $\lambda'=(\lambda+t^2)/s^2$, $t\in\F$, $s\in\F^\times$ & \cite{DerNonNilpotent3} & previously known\\
$L_9$ & $[a_1,a_1]=[a_1,a_3]=a_3$ & no complete arbitrary-field description located; related calculation in \cite{DerNonNilpotent4} & complete formula given here\\
$L_{10}(r)$ & $r\in\Fx$; $[a_1,a_1]=a_3$, $[a_1,a_2]=a_2$, $[a_2,a_1]=-a_2$, $[a_1,a_3]=ra_3$; different values of $r$ give different normalized types & no complete arbitrary-field description located for corresponding presentations & all $r$ and all characteristics here\\
$L_{11}$ & $\operatorname{char}\F\ne2$; $[a_2,a_2]=a_3$, $[a_1,a_2]=a_2$, $[a_2,a_1]=-a_2$, $[a_1,a_3]=2a_3$ & no direct arbitrary-field formula used here & direct formula in present basis\\
$L_{12}$ & $[a_1,a_1]=a_2$, $[a_1,a_2]=a_3$ & \cite{DerCyclic1,DerCyclic2,DerLowDimsUMJ} & previously known\\
$L_{13}$ & $[a_1,a_1]=a_2$, $[a_1,a_2]=a_2+a_3$ & \cite{DerCyclic1,DerNonNilpotent1,DerSome} & previously known\\
$L_{14}(\rho)$ & $\rho\in\Fx$; $[a_1,a_1]=[a_1,a_2]=a_2$, $[a_1,a_3]=\rho a_3$; $\rho\sim\rho^{-1}$ & $\rho=1$ in \cite{DerNonNilpotent1}; no complete arbitrary-field description located & general parameter completed here\\
$L_{15}$ & $[a_1,a_1]=a_2$, $[a_1,a_2]=a_2+a_3$, $[a_1,a_3]=a_3$ & \cite{DerCyclic1,DerNonNilpotent1} & previously known, including characteristic $2$\\
$L_{16}(\beta,\gamma)$ & $[a_1,a_1]=a_2$, $[a_1,a_2]=a_3$, $[a_1,a_3]=\beta a_2+\gamma a_3$; $X^2-\gamma X-\beta$ irreducible; $(\beta,\gamma)\sim$ $(c^2\beta,c\gamma)$, $c\in\F^\times$ & related calculations in \cite{DerCyclic1,DerSome} & explicit arbitrary-field formula here\\
\end{longtable}
\endgroup

\begin{remark}[Changes of basis in known cases]\label{rem:basis}
Some earlier papers use representatives different from the ones fixed above.  We record the changes that will be used later.

For $L_2$, the presentation
\[
 [e_1,e_1]=[e_2,e_1]=e_3
\]
is changed to our table by
\[
 a_1=e_1,\qquad a_2=e_1-e_2,\qquad a_3=e_3.
\]

A representative used for $L_7$ is
\[
 [e_1,e_1]=e_3,\qquad [e_1,e_2]=-e_2,\qquad [e_2,e_1]=e_2.
\]
The change $a_1=-e_1$, $a_2=e_2$, $a_3=e_3$ gives the table used here.

A common presentation leading to $L_{13}$ is
\[
 [e_1,e_1]=e_3,\qquad [e_1,e_2]=e_2+\lambda e_3.
\]
For any $\lambda\in\F$, the basis
\[
 a_1=e_1+e_2,\qquad
 a_2=e_2+(\lambda+1)e_3,\qquad
 a_3=-e_3
\]
gives $[a_1,a_1]=a_2$ and $[a_1,a_2]=a_2+a_3$.

For $L_{14}(1)$, the presentation
\[
 [e_1,e_1]=[e_1,e_3]=e_3,\qquad [e_1,e_2]=e_2
\]
is changed to our table by $a_1=e_1$, $a_2=e_3$, $a_3=e_2$.

Finally, for
\[
 [e_1,e_1]=[e_1,e_3]=e_3,\qquad [e_1,e_2]=e_2+\lambda e_3,
 \qquad \lambda\ne0,
\]
the basis
\[
 a_1=e_1+e_2-(1+\lambda)e_3,\qquad
 a_2=e_2,\qquad
 a_3=\lambda e_3
\]
gives $L_{15}$.
\end{remark}

\section{Nilpotent algebras with one-dimensional Leibniz kernel}

The derivation algebras of $L_1$ and $L_2$ are already available in the literature cited in Table~\ref{tab:status}.  In our basis they are
\[
 \Der(L_1)=
 \left\{
 \begin{pmatrix}
 a&0&0\\ b&d&0\\ c&e&2a
 \end{pmatrix}
 :a,b,c,d,e\in\F
 \right\},
\]
and
\[
 \Der(L_2)=
 \left\{
 \begin{pmatrix}
 a&0&0\\ b&a+b&0\\ c&d&2a+b
 \end{pmatrix}
 :a,b,c,d\in\F
 \right\}.
\]
Thus $\dim\Der(L_1)=5$ and $\dim\Der(L_2)=4$.

For $L_1$, put $D=\Der(L_1)$, and let $u,v,w,z,t\in D$ be the derivations obtained from the displayed matrix by setting, respectively,
\[
 a=1,\qquad b=1,\qquad c=1,\qquad d=1,\qquad e=1,
\]
with all remaining parameters equal to zero.  Then $u,v,w,z,t$ form a basis of $D$.  The subspace
\[
 \F v\oplus\F w\oplus\F t
\]
is an ideal of $D$, $\F u\oplus\F z$ is an abelian subalgebra, and
\[
 D=(\F v\oplus\F w\oplus\F t)\rtimes(\F u\oplus\F z).
\]
Inside the first ideal the only nonzero bracket, up to skew-symmetry, is
\[
 [v,t]=-w.
\]
The action of $\F u\oplus\F z$ is given by
\[
 [u,v]=-v,\quad [u,w]=w,\quad [u,t]=2t,
 \qquad [z,v]=v,\quad [z,t]=-t,
\]
and $[z,w]=0$.

For $L_2$, put $D=\Der(L_2)$, and let $u,v,w,z\in D$ correspond, respectively, to the parameters $a,b,c,d$ in the displayed matrix.  Thus $u,v,w,z$ form a basis of $D$.  The subspace $\F w\oplus\F z$ is an abelian ideal, $\F u\oplus\F v$ is an abelian subalgebra, and
\[
 D=(\F w\oplus\F z)\rtimes(\F u\oplus\F v).
\]
The nonzero action brackets are
\[
 [u,w]=w,\qquad [u,z]=z,\qquad [v,w]=w,\qquad [v,z]=-w.
\]

We now turn to the remaining nilpotent types.

\begin{theorem}\label{thm:L3}
Let $L=L_3(\alpha)$, where $\alpha\ne0,-1$.  Then
\[
\Der(L)=
\left\{
\begin{pmatrix}
 a&0&0\\
 b&a+(1+\alpha)b&0\\
 c&d&2a+(1+\alpha)b
\end{pmatrix}
:a,b,c,d\in\F
\right\}.
\]
In particular, $\dim_{\F}\Der(L)=4$.  Put $D=\Der(L)$, and let $u,v,w,z\in D$ be the derivations corresponding, respectively, to the parameters $a,b,c,d$ in the displayed matrix.  Then $u,v,w,z$ form a basis of $D$.  The subspace $\F w\oplus\F z$ is an abelian ideal, $\F u\oplus\F v$ is an abelian subalgebra, and
\[
 D=(\F w\oplus\F z)\rtimes(\F u\oplus\F v).
\]
The action is
\[
 [u,w]=w,\qquad [u,z]=z,\qquad
 [v,w]=(1+\alpha)w,\qquad [v,z]=-w.
\]
\end{theorem}

\begin{proof}
Let $d\in\Der(L)$.  Since $\Leib(L)=\F a_3$, Lemma~\ref{lem:invariant} gives
\[
\begin{aligned}
 d(a_1)&=a a_1+b a_2+c a_3,\\
 d(a_2)&=p a_1+q a_2+d a_3,\\
 d(a_3)&=s a_3.
\end{aligned}
\]
We use the defining products one by one.

From $[a_1,a_1]=a_3$ we obtain
\[
\begin{aligned}
 s a_3
 &= [d(a_1),a_1]+[a_1,d(a_1)]\\
 &= (a+b)a_3+(a+\alpha b)a_3.
\end{aligned}
\]
Hence
\begin{equation}\label{eq:L3s}
 s=2a+(1+\alpha)b.
\end{equation}

Since $[a_2,a_2]=0$,
\[
 0=[d(a_2),a_2]+[a_2,d(a_2)]
   =(1+\alpha)p a_3.
\]
The assumption $\alpha\ne-1$ gives
\begin{equation}\label{eq:L3p}
 p=0.
\end{equation}

Now apply $d$ to $[a_2,a_1]=a_3$.  Using $p=0$, we have
\[
 s a_3=[q a_2+d a_3,a_1]+[a_2,a a_1+b a_2+c a_3]
       =(q+a)a_3.
\]
Therefore
\[
 q=s-a=a+(1+\alpha)b.
\]
Finally, the relation $[a_1,a_2]=\alpha a_3$ gives the same equality and no new condition.  Products involving $a_3$ are zero on both sides.  This proves that every derivation has the displayed form.

Conversely, take a linear map with the displayed matrix.  Substitution in the three nonzero products and in $[a_2,a_2]=0$ gives the equalities above in reverse order.  Since every product involving $a_3$ is zero, the derivation identity holds for all basis pairs.  Hence the displayed set is exactly $\Der(L)$.  The structural assertions follow by taking the commutators of the basis derivations $u,v,w,z$ defined in the theorem.
\end{proof}

\begin{theorem}\label{thm:L4}
Let $L=L_4$.

If $\operatorname{char}\F\ne2$, then
\[
\Der(L)=
\left\{
\begin{pmatrix}
 a&0&0\\
 b&a&0\\
 c&d&2a
\end{pmatrix}
:a,b,c,d\in\F
\right\},
\qquad \dim\Der(L)=4.
\]

If $\operatorname{char}\F=2$, then
\[
\Der(L)=
\left\{
\begin{pmatrix}
 a&p&0\\
 b&a+p&0\\
 c&d&0
\end{pmatrix}
:a,b,c,d,p\in\F
\right\},
\qquad \dim\Der(L)=5.
\]

If $\operatorname{char}\F\ne2$, put $D=\Der(L)$ and let $u,v,w,z\in D$ correspond, respectively, to $a,b,c,d$ in the first displayed matrix.  Then $u,v,w,z$ form a basis of $D$, $\F w\oplus\F z$ is an abelian ideal, $\F u\oplus\F v$ is an abelian subalgebra, and
\[
 D=(\F w\oplus\F z)\rtimes(\F u\oplus\F v),
\]
where
\[
 [u,w]=w,\qquad [u,z]=z,\qquad [v,z]=-w,
\]
and the remaining brackets between these two summands are zero.

If $\operatorname{char}\F=2$, let $u,v,w,z,t\in D$ correspond, respectively, to the parameters $a,p,b,c,d$ in the second displayed matrix.  Then $u,v,w,z,t$ form a basis of $D$.  The subspace $\F z\oplus\F t$ is an abelian ideal.  If
\[
 H=\F u\oplus\F v\oplus\F w,
\]
then $H$ is a subalgebra and
\[
 D=(\F z\oplus\F t)\rtimes H.
\]
Inside $H$ we have
\[
 [v,w]=u+w.
\]
Moreover, $\F u$ is central in $H$ and
\[
 H=(\F(u+w)\rtimes\F v)\oplus\F u,
 \qquad [v,u+w]=u+w.
\]
The action on the ideal $\F z\oplus\F t$ is
\[
 [u,z]=z,\quad [u,t]=t,\quad [v,z]=t,\quad
 [v,t]=t,\quad [w,t]=z.
\]
\end{theorem}

\begin{proof}
Write
\[
\begin{aligned}
 d(a_1)&=a a_1+b a_2+c a_3,\\
 d(a_2)&=p a_1+q a_2+d a_3,\\
 d(a_3)&=s a_3.
\end{aligned}
\]
From $[a_1,a_1]=a_3$ we get
\[
 s=2a.
\]
The relation $[a_2,a_2]=0$ gives no condition in this case.  From $[a_2,a_1]=a_3$ we obtain
\[
 s=p+q+a,
\]
while $[a_1,a_2]=-a_3$ gives
\[
 -s=-a+p-q.
\]
Substituting $s=2a$ into these two equalities gives
\[
 q=a-p,\qquad q=a+p.
\]
Hence
\[
 2p=0.
\]
If $\operatorname{char}\F\ne2$, then $p=0$, $q=a$, and the first matrix follows.  If $\operatorname{char}\F=2$, the parameter $p$ is free, $q=a+p$, and $s=0$.  This gives the second matrix.

Conversely, direct substitution in
\[
 [a_1,a_1]=[a_2,a_1]=a_3,
 \qquad [a_1,a_2]=-a_3,
 \qquad [a_2,a_2]=0
\]
shows that every displayed matrix defines a derivation.  Thus no additional condition is required.  The decompositions and the listed commutators follow by direct multiplication of the matrices of the basis derivations defined in the theorem.
\end{proof}

\begin{theorem}\label{thm:L5}
Let $L=L_5(\beta)$.  Then $\beta\ne0$.

If $\operatorname{char}\F\ne2$, then
\[
\Der(L)=
\left\{
\begin{pmatrix}
 a&-\beta b&0\\
 b&a&0\\
 c&d&2a
\end{pmatrix}
:a,b,c,d\in\F
\right\},
\qquad \dim\Der(L)=4.
\]

If $\operatorname{char}\F=2$, then
\[
\Der(L)=
\left\{
\begin{pmatrix}
 a&\beta b&0\\
 b&e&0\\
 c&d&0
\end{pmatrix}
:a,b,c,d,e\in\F
\right\},
\qquad \dim\Der(L)=5.
\]

If $\operatorname{char}\F\ne2$, put $D=\Der(L)$ and let $u,v,w,z\in D$ correspond, respectively, to $a,b,c,d$ in the first displayed matrix.  Then $u,v,w,z$ form a basis of $D$.  The subspace $\F w\oplus\F z$ is an abelian ideal, $\F u\oplus\F v$ is an abelian subalgebra, and
\[
 D=(\F w\oplus\F z)\rtimes(\F u\oplus\F v),
\]
with
\[
 [u,w]=w,\qquad [u,z]=z,\qquad
 [v,w]=\beta z,\qquad [v,z]=-w.
\]

If $\operatorname{char}\F=2$, let $u,v,w,z,t\in D$ correspond, respectively, to the parameters $a,b,c,d,e$ in the second displayed matrix.  Then $u,v,w,z,t$ form a basis of $D$.  The subspace $\F w\oplus\F z$ is an abelian ideal.  If
\[
 H=\F u\oplus\F v\oplus\F t,
\]
then $H$ is a subalgebra and
\[
 D=(\F w\oplus\F z)\rtimes H.
\]
Inside $H$,
\[
 [u,v]=v,\qquad [t,v]=v,
\]
so $u+t$ is central in $H$ and
\[
 H=(\F v\rtimes\F u)\oplus\F(u+t).
\]
The action on $\F w\oplus\F z$ is
\[
 [u,w]=w,\qquad [t,z]=z,\qquad
 [v,w]=\beta z,\qquad [v,z]=w.
\]
\end{theorem}

\begin{proof}
The polynomial $X^2+\beta$ has no root, so $\beta\ne0$.  Put
\[
\begin{aligned}
 d(a_1)&=a a_1+b a_2+c a_3,\\
 d(a_2)&=p a_1+q a_2+d a_3,\\
 d(a_3)&=s a_3.
\end{aligned}
\]
From $[a_1,a_1]=a_3$ we get
\begin{equation}\label{eq:L5a}
 s=2a.
\end{equation}
From $[a_2,a_2]=\beta a_3$ we obtain
\[
 \beta s=2\beta q,
\]
and hence
\begin{equation}\label{eq:L5q}
 2(a-q)=0.
\end{equation}
The zero product $[a_1,a_2]=0$ gives
\[
 0=[d(a_1),a_2]+[a_1,d(a_2)]
   =(\beta b+p)a_3,
\]
so
\begin{equation}\label{eq:L5p}
 p=-\beta b.
\end{equation}
The relation $[a_2,a_1]=0$ gives the same condition.

If $\operatorname{char}\F\ne2$, equation~\eqref{eq:L5q} gives $q=a$, and we obtain the first family.  If $\operatorname{char}\F=2$, then $s=0$, the coefficient $q$ is free, and $-\beta b=\beta b$.  Renaming $q$ as $e$ gives the second family.

For the converse, the two square relations and the two zero mixed products have already been checked in the calculation above.  All products involving $a_3$ are zero.  Hence every displayed matrix is a derivation.  Direct multiplication of the matrices of the basis derivations gives the decompositions and commutators stated in the theorem.
\end{proof}

\begin{remark}
A calculation for the same general nilpotent situation appears in~\cite{DerLowDims}.  For the present representative, however, the full arbitrary-field coefficient relations are most transparent when they are written directly from the defining products.  This is the reason for including the complete calculation above.
\end{remark}

\begin{theorem}\label{thm:L6}
Let $L=L_6(\delta)$.  Then
\[
\Der(L)=
\left\{
\begin{pmatrix}
 a&-\delta b&0\\
 b&a+b&0\\
 c&d&2a+b
\end{pmatrix}
:a,b,c,d\in\F
\right\}.
\]
In particular, $\dim_{\F}\Der(L)=4$ in every characteristic.  Put $D=\Der(L)$, and let $u,v,w,z\in D$ correspond, respectively, to the parameters $a,b,c,d$ in the displayed matrix.  Then $u,v,w,z$ form a basis of $D$.  The subspace $\F w\oplus\F z$ is an abelian ideal, $\F u\oplus\F v$ is an abelian subalgebra, and
\[
 D=(\F w\oplus\F z)\rtimes(\F u\oplus\F v).
\]
The action is
\[
 [u,w]=w,\qquad [u,z]=z,\qquad
 [v,w]=w+\delta z,\qquad [v,z]=-w.
\]
\end{theorem}

\begin{proof}
Let
\[
\begin{aligned}
 d(a_1)&=a a_1+b a_2+c a_3,\\
 d(a_2)&=p a_1+q a_2+d a_3,\\
 d(a_3)&=s a_3.
\end{aligned}
\]
Applying $d$ to $[a_1,a_1]=a_3$ gives
\begin{equation}\label{eq:L6s}
 s=2a+b.
\end{equation}
Indeed,
\[
 [d(a_1),a_1]=a a_3,
 \qquad
 [a_1,d(a_1)]=(a+b)a_3.
\]

Next, $[a_2,a_1]=0$ gives
\[
 0=[d(a_2),a_1]+[a_2,d(a_1)]
   =(p+\delta b)a_3,
\]
so
\begin{equation}\label{eq:L6p}
 p=-\delta b.
\end{equation}
Using $[a_1,a_2]=a_3$, we obtain
\[
 s=a+\delta b+p+q.
\]
Together with~\eqref{eq:L6s} and~\eqref{eq:L6p}, this gives
\begin{equation}\label{eq:L6q}
 q=a+b.
\end{equation}

It remains to check the square of $a_2$.  We have
\[
\begin{aligned}
 [d(a_2),a_2]+[a_2,d(a_2)]
 &=p a_3+2\delta q a_3\\
 &=\delta(2a+b)a_3
 =\delta d(a_3),
\end{aligned}
\]
where we used~\eqref{eq:L6p} and~\eqref{eq:L6q}.  Thus no further condition occurs.

Conversely, a map with the displayed matrix satisfies the derivation identity for $[a_1,a_1]$, $[a_1,a_2]$, $[a_2,a_2]$, and $[a_2,a_1]=0$ by the same equalities.  All products involving $a_3$ vanish.  Hence it is a derivation.  The final decomposition is obtained by taking commutators of the matrices of the basis derivations defined in the theorem.
\end{proof}

\begin{remark}
The algebra $L_6(\delta)$ is also related to a family considered in~\cite{DerLowDims}.  The direct calculation above is included because it gives one formula valid in every characteristic and in the present normalization.
\end{remark}

\section{The remaining cases with one-dimensional Leibniz kernel}

The derivation algebras of $L_7$ and $L_8(\lambda)$ are already contained in~\cite{DerNonNilpotent3}; the changes of presentation are described in Remark~\ref{rem:basis}.  In our basis,
\[
 \Der(L_7)=
 \left\{
 \begin{pmatrix}
 0&0&0\\ a&b&0\\ c&0&0
 \end{pmatrix}:a,b,c\in\F
 \right\},
 \qquad \dim\Der(L_7)=3,
\]
and, when $\operatorname{char}\F=2$,
\[
 \Der(L_8(\lambda))=
 \left\{
 \begin{pmatrix}
 0&0&0\\ a&b&0\\ c&a&0
 \end{pmatrix}:a,b,c\in\F
 \right\},
 \qquad \dim\Der(L_8(\lambda))=3.
\]
In particular, the derivation algebra of $L_8(\lambda)$ does not depend on $\lambda$.  For either $L=L_7$ or $L=L_8(\lambda)$, put $D=\Der(L)$ and let $u,v,w\in D$ correspond, respectively, to the parameters $a,b,c$ in the displayed matrix.  Then $u,v,w$ form a basis of $D$,
\[
 [v,u]=u,\qquad w\in\zeta(D),
\]
and therefore
\[
 D=(\F u\rtimes\F v)\oplus\F w.
\]

\begin{theorem}\label{thm:L9}
Let $L=L_9$.  Then
\[
\Der(L)=
\left\{
\begin{pmatrix}
0&0&0\\
a&b&0\\
c&0&c
\end{pmatrix}
:a,b,c\in\F
\right\}.
\]
Thus $\dim_{\F}\Der(L)=3$ in every characteristic.  Put $D=\Der(L)$, and let $u,v,w\in D$ correspond, respectively, to the parameters $a,b,c$ in the displayed matrix.  Then $u,v,w$ form a basis of $D$,
\[
 [v,u]=u,\qquad w\in\zeta(D),
\]
and hence
\[
 D=(\F u\rtimes\F v)\oplus\F w.
\]
\end{theorem}

\begin{proof}
Here
\[
 \Leib(L)=[L,L]=\F a_3,
 \qquad
 \zeta(L)=\F a_2.
\]
By Lemma~\ref{lem:invariant}, write
\[
 d(a_1)=x a_1+a a_2+u a_3,
 \qquad
 d(a_2)=b a_2,
 \qquad
 d(a_3)=c a_3.
\]
Apply $d$ to $[a_1,a_3]=a_3$.  We obtain
\[
 c a_3=[d(a_1),a_3]+[a_1,d(a_3)]
        =(x+c)a_3,
\]
so $x=0$.  Now use $[a_1,a_1]=a_3$:
\[
 c a_3=[a a_2+u a_3,a_1]+[a_1,a a_2+u a_3]
       =u a_3.
\]
Hence $u=c$.  This gives exactly the displayed matrix.

Conversely, for such a matrix,
\[
 d([a_1,a_1])=d(a_3)=c a_3
 =[d(a_1),a_1]+[a_1,d(a_1)],
\]
and
\[
 d([a_1,a_3])=c a_3
 =[d(a_1),a_3]+[a_1,d(a_3)].
\]
All remaining products are zero and satisfy the derivation identity immediately from the displayed images.  Therefore the formula is complete.  The commutator calculation for the three basis derivations gives $[v,u]=u$ and shows that $w$ is central, proving the structural assertion.
\end{proof}

\begin{remark}\label{rem:L9old}
A related calculation for this multiplication appears in~\cite{DerNonNilpotent4}.  When its matrix family is compared with the present basis, the coefficient of $a_2$ in $d(a_1)$ is not present.  This coefficient is in fact free: for example,
\[
 d(a_1)=a_2,\qquad d(a_2)=d(a_3)=0
\]
is a derivation of $L_9$.  Theorem~\ref{thm:L9} therefore records the full family in the present basis.
\end{remark}

\begin{theorem}\label{thm:L10}
Let $L=L_{10}(r)$, where $r\in\Fx$.  Then
\[
\Der(L)=
\left\{
\begin{pmatrix}
0&0&0\\
a&b&0\\
c&0&rc
\end{pmatrix}
:a,b,c\in\F
\right\}.
\]
Hence $\dim_{\F}\Der(L)=3$.  The dimension and the number of free parameters are independent of $r$.  Put $D=\Der(L)$, and let $u,v,w\in D$ correspond, respectively, to $a,b,c$.  Then $u,v,w$ form a basis of $D$,
\[
 [v,u]=u,\qquad w\in\zeta(D),
\]
and
\[
 D=(\F u\rtimes\F v)\oplus\F w.
\]
In particular, the Lie-algebra structure of $\Der(L_{10}(r))$ is independent of $r\ne0$.
\end{theorem}

\begin{proof}
Since $\Leib(L)=\F a_3$ and $[L,L]=\F a_2\oplus\F a_3$, write
\[
\begin{aligned}
 d(a_1)&=x a_1+a a_2+c a_3,\\
 d(a_2)&=b a_2+u a_3,\\
 d(a_3)&=v a_3.
\end{aligned}
\]
From $[a_1,a_3]=r a_3$ we get
\[
 rv a_3=[d(a_1),a_3]+[a_1,d(a_3)]
        =r(x+v)a_3.
\]
Since $r\ne0$, we have
\begin{equation}\label{eq:L10x}
 x=0.
\end{equation}

Next apply $d$ to $[a_2,a_1]=-a_2$.  Using $x=0$,
\[
 -b a_2-u a_3
 =[b a_2+u a_3,a_1]+[a_2,a a_2+c a_3]
 =-b a_2.
\]
Thus
\begin{equation}\label{eq:L10u}
 u=0.
\end{equation}
Finally, from $[a_1,a_1]=a_3$,
\[
 v a_3=[a a_2+c a_3,a_1]+[a_1,a a_2+c a_3]
       =rc a_3,
\]
so $v=rc$.

The relation $[a_1,a_2]=a_2$ is then preserved without any additional condition.  Conversely, direct substitution in the four nonzero products proves that every displayed matrix is a derivation.  Direct multiplication of the matrices of $u,v,w$ gives the stated decomposition.
\end{proof}

\begin{theorem}\label{thm:L11}
Let $L=L_{11}$, so $\operatorname{char}\F\ne2$.  Then
\[
\Der(L)=
\left\{
\begin{pmatrix}
0&0&0\\
a&b&0\\
0&-a&2b
\end{pmatrix}
:a,b\in\F
\right\}.
\]
In particular, $\dim_{\F}\Der(L)=2$.  Put $D=\Der(L)$, and let $u,v\in D$ correspond, respectively, to $a,b$.  Then $u,v$ form a basis of $D$ and
\[
 [v,u]=u.
\]
Consequently,
\[
 D=\F u\rtimes\F v.
\]
\end{theorem}

\begin{proof}
Since $\Leib(L)=\F a_3$ and $[L,L]=\F a_2\oplus\F a_3$, write
\[
\begin{aligned}
 d(a_1)&=x a_1+a a_2+c a_3,\\
 d(a_2)&=b a_2+u a_3,\\
 d(a_3)&=v a_3.
\end{aligned}
\]
Apply $d$ to $[a_1,a_3]=2a_3$.  We get
\[
 2v a_3=2(x+v)a_3.
\]
Since the characteristic is not two,
\begin{equation}\label{eq:L11x}
 x=0.
\end{equation}

From $[a_2,a_1]=-a_2$ we obtain
\[
 -b a_2-u a_3
 =[b a_2+u a_3,a_1]+[a_2,a a_2+c a_3]
 =-b a_2+a a_3,
\]
so
\begin{equation}\label{eq:L11u}
 u=-a.
\end{equation}
The square $[a_2,a_2]=a_3$ gives
\[
 v a_3=2b a_3,
\]
so $v=2b$.

It remains to use $[a_1,a_1]=0$.  We have
\[
 0=[a a_2+c a_3,a_1]+[a_1,a a_2+c a_3]
   =2c a_3.
\]
Hence $c=0$.  The relation $[a_1,a_2]=a_2$ now gives no new condition.  This proves the necessity of the displayed form, and the converse follows by substituting it in the defining products.  The bracket $[v,u]=u$ follows immediately from the matrices of the basis derivations defined in the theorem.
\end{proof}

\section{Algebras with two-dimensional Leibniz kernel}

The nilpotent one-generated algebra $L_{12}$ is covered by the general one-generated results~\cite{DerCyclic1,DerCyclic2}.  In our basis,
\[
\Der(L_{12})=
\left\{
\begin{pmatrix}
 a&0&0\\
 b&2a&0\\
 c&b&3a
\end{pmatrix}:a,b,c\in\F
\right\},
\qquad \dim\Der(L_{12})=3.
\]
This formula is valid in every characteristic; the coefficients $2$ or $3$ may of course vanish in the corresponding characteristic.  Put $D=\Der(L_{12})$, and let $u,v,w\in D$ correspond, respectively, to the parameters $a,b,c$.  Then $u,v,w$ form a basis of $D$, $\F v\oplus\F w$ is an abelian ideal, and
\[
 D=(\F v\oplus\F w)\rtimes\F u,
 \qquad [u,v]=v,\qquad [u,w]=2w.
\]
In characteristic two, $\F w$ is central.

For $L_{13}$, using the change of basis in Remark~\ref{rem:basis} and the results of~\cite{DerNonNilpotent1}, we obtain
\[
\Der(L_{13})=
\left\{
\begin{pmatrix}
0&0&0\\
a&a&0\\
b&a&0
\end{pmatrix}:a,b\in\F
\right\},
\qquad \dim\Der(L_{13})=2.
\]
Put $D=\Der(L_{13})$, and let $u,v\in D$ correspond, respectively, to the parameters $a,b$.  Then $u,v$ form a basis of $D$, $[u,v]=0$, and hence
\[
 D=\F u\oplus\F v
\]
is abelian.

We next consider $L_{14}(\rho)$.

\begin{theorem}\label{thm:L14}
Let $L=L_{14}(\rho)$, where $\rho\in\Fx$.

If $\rho\ne1$, then
\[
\Der(L)=
\left\{
\begin{pmatrix}
0&0&0\\
a&a&0\\
0&0&b
\end{pmatrix}:a,b\in\F
\right\},
\qquad \dim\Der(L)=2.
\]

If $\rho=1$, then
\[
\Der(L)=
\left\{
\begin{pmatrix}
0&0&0\\
a&a&c\\
b&b&d
\end{pmatrix}:a,b,c,d\in\F
\right\},
\qquad \dim\Der(L)=4.
\]
The second formula agrees, after the change of basis in Remark~\ref{rem:basis}, with the corresponding result of~\cite{DerNonNilpotent1}.

If $\rho\ne1$, put $D=\Der(L)$ and let $u,v\in D$ correspond, respectively, to the parameters $a,b$ in the first displayed matrix.  Then $u,v$ form a basis of $D$, $[u,v]=0$, and
\[
 D=\F u\oplus\F v
\]
is abelian.

If $\rho=1$, for
\[
 B=\begin{pmatrix}a&c\\ b&d\end{pmatrix}\in M_2(\F)
\]
let $d_B\in D$ be the derivation whose matrix is
\[
 \Xi(d_B)=\begin{pmatrix}0&0&0\\ a&a&c\\ b&b&d\end{pmatrix}.
\]
Then the map $B\mapsto d_B$ is a Lie-algebra isomorphism
\[
 \mathfrak{gl}_2(\F)\cong D.
\]
\end{theorem}

\begin{proof}
We prove the general calculation, which also shows why $\rho=1$ is exceptional.  Since $\Leib(L)=\F a_2\oplus\F a_3$, write
\[
\begin{aligned}
 d(a_1)&=s a_1+a a_2+b a_3,\\
 d(a_2)&=p a_2+q a_3,\\
 d(a_3)&=u a_2+v a_3.
\end{aligned}
\]
From $[a_1,a_1]=a_2$ we obtain
\begin{equation}\label{eq:L14square}
 p=2s+a,
 \qquad
 q=\rho b.
\end{equation}
Now use $[a_1,a_2]=a_2$.  We get
\[
 p a_2+q a_3
 =(s+p)a_2+\rho q a_3.
\]
Hence
\begin{equation}\label{eq:L14s}
 s=0,
 \qquad
 (\rho-1)q=0.
\end{equation}
Since $\rho\ne0$, equations~\eqref{eq:L14square} and~\eqref{eq:L14s} give
\[
 p=a,
 \qquad
 (\rho-1)b=0.
\]
Finally, apply $d$ to $[a_1,a_3]=\rho a_3$:
\[
 \rho(u a_2+v a_3)=u a_2+\rho v a_3,
\]
so
\begin{equation}\label{eq:L14u}
 (\rho-1)u=0.
\end{equation}

If $\rho\ne1$, then $b=u=0$, and we obtain the first matrix.  If $\rho=1$, both $b$ and $u$ remain free, and the second matrix follows.  Conversely, substitution in the three defining products proves that every displayed matrix is a derivation.

For $\rho\ne1$, direct multiplication shows that the two basis derivations commute.  For $\rho=1$, the displayed matrices satisfy $\Xi(d_B)\Xi(d_C)=\Xi(d_{BC})$.  Hence $[d_B,d_C]=d_{BC-CB}$, which proves the stated isomorphism with $\mathfrak{gl}_2(\F)$.
\end{proof}

For $L_{15}$ the result in~\cite{DerNonNilpotent1}, together with the last change of basis in Remark~\ref{rem:basis}, gives the following two formulas.  If $\operatorname{char}\F\ne2$, then
\[
\Der(L_{15})=
\left\{
\begin{pmatrix}
0&0&0\\
a&a&0\\
b&a+b&a
\end{pmatrix}:a,b\in\F
\right\},
\qquad \dim\Der(L_{15})=2.
\]
If $\operatorname{char}\F=2$, then
\[
\Der(L_{15})=
\left\{
\begin{pmatrix}
s&0&0\\
a&a&s\\
b&a+b&s+a
\end{pmatrix}:s,a,b\in\F
\right\},
\qquad \dim\Der(L_{15})=3.
\]
Thus characteristic two produces one additional free parameter.  If $\operatorname{char}\F\ne2$, put $D=\Der(L_{15})$ and let $u,v\in D$ correspond, respectively, to $a,b$.  Then $u,v$ form a basis of $D$, $[u,v]=0$, and
\[
 D=\F u\oplus\F v
\]
is abelian.

If $\operatorname{char}\F=2$, let $u,v,w\in D$ correspond, respectively, to the parameters $s,a,b$ in the second displayed matrix.  Then $u,v,w$ form a basis of $D$,
\[
 [u,v]=v,
\]
and $v+w$ is central.  Therefore
\[
 D=(\F v\rtimes\F u)\oplus\F(v+w).
\]

We finish with the irreducible family $L_{16}(\beta,\gamma)$.

\begin{theorem}\label{thm:L16}
Let $L=L_{16}(\beta,\gamma)$, where $X^2-\gamma X-\beta$ is irreducible over $\F$.  In particular, $\beta\ne0$.

If $\operatorname{char}\F\ne2$, or if $\gamma\ne0$, then
\[
\Der(L)=
\left\{
\begin{pmatrix}
0&0&0\\
a&\beta b&\beta(a+\gamma b)\\
b&a+\gamma b&\gamma a+(\beta+\gamma^2)b
\end{pmatrix}
:a,b\in\F
\right\}.
\]
In this case $\dim_{\F}\Der(L)=2$.

If $\operatorname{char}\F=2$ and $\gamma=0$, then
\[
\Der(L)=
\left\{
\begin{pmatrix}
s&0&0\\
a&\beta b&\beta a\\
b&a&s+\beta b
\end{pmatrix}
:s,a,b\in\F
\right\}.
\]
In this case $\dim_{\F}\Der(L)=3$.

In the first case put $D=\Der(L)$ and let $u,v\in D$ correspond, respectively, to the parameters $a,b$ in the first displayed matrix.  Then $u,v$ form a basis of $D$, $[u,v]=0$, and
\[
 D=\F u\oplus\F v
\]
is abelian.

In the exceptional case $\operatorname{char}\F=2$, $\gamma=0$, let $u,v,w\in D$ correspond, respectively, to the parameters $s,a,b$ in the second displayed matrix.  Then $u,v,w$ form a basis of $D$,
\[
 [u,v]=v,\qquad w\in\zeta(D),
\]
and hence
\[
 D=(\F v\rtimes\F u)\oplus\F w.
\]
\end{theorem}

\begin{proof}
Irreducibility gives $\beta\ne0$, because otherwise $X^2-\gamma X=X(X-\gamma)$.

Let $d\in\Der(L)$.  Since
\[
 \Leib(L)=[L,L]=\F a_2\oplus\F a_3,
\]
we may write
\[
\begin{aligned}
 d(a_1)&=s a_1+a a_2+b a_3,\\
 d(a_2)&=p a_2+q a_3,\\
 d(a_3)&=u a_2+v a_3.
\end{aligned}
\]
We now use the three defining products.

First, $[a_1,a_1]=a_2$ gives
\begin{equation}\label{eq:L16pq}
 p=2s+\beta b,
 \qquad
 q=a+\gamma b.
\end{equation}
Indeed,
\[
\begin{aligned}
[d(a_1),a_1]+[a_1,d(a_1)]
&=s a_2+s a_2+a a_3+b(\beta a_2+\gamma a_3)\\
&=(2s+\beta b)a_2+(a+\gamma b)a_3.
\end{aligned}
\]

Next, $[a_1,a_2]=a_3$ gives
\begin{equation}\label{eq:L16uv}
 u=\beta q,
 \qquad
 v=s+p+\gamma q.
\end{equation}
Using~\eqref{eq:L16pq}, this becomes
\[
 u=\beta(a+\gamma b),
\]
and
\begin{equation}\label{eq:L16v}
 v=3s+\gamma a+(\beta+\gamma^2)b.
\end{equation}

Finally, apply $d$ to
\[
 [a_1,a_3]=\beta a_2+\gamma a_3.
\]
The coefficient of $a_2$ gives
\[
 2\beta s=0,
\]
and the coefficient of $a_3$ gives
\[
 \gamma s=0.
\]
Since $\beta\ne0$, we have
\begin{equation}\label{eq:L16exception}
 2s=0,
 \qquad
 \gamma s=0.
\end{equation}

If $\operatorname{char}\F\ne2$, then $s=0$.  If $\operatorname{char}\F=2$ but $\gamma\ne0$, again $s=0$.  Substituting $s=0$ in~\eqref{eq:L16pq}--\eqref{eq:L16v} gives the first matrix.

The only remaining possibility is
\[
 \operatorname{char}\F=2,
 \qquad
 \gamma=0.
\]
Then~\eqref{eq:L16exception} imposes no condition on $s$, and
\[
 p=\beta b,
 \qquad
 q=a,
 \qquad
 u=\beta a,
 \qquad
 v=s+\beta b.
\]
This is exactly the second matrix.

It remains to verify the converse.  For the first family, set
\[
\begin{aligned}
 d(a_1)&=a a_2+b a_3,\\
 d(a_2)&=\beta b a_2+(a+\gamma b)a_3,\\
 d(a_3)&=\beta(a+\gamma b)a_2+
 \bigl(\gamma a+(\beta+\gamma^2)b\bigr)a_3.
\end{aligned}
\]
Since $[\F a_2+\F a_3,L]=0$, we obtain
\[
\begin{aligned}
[d(a_1),a_1]+[a_1,d(a_1)]
 &=a[a_1,a_2]+b[a_1,a_3]\\
 &=\beta b a_2+(a+\gamma b)a_3=d(a_2),
\end{aligned}
\]
and
\[
\begin{aligned}
[d(a_1),a_2]+[a_1,d(a_2)]
 &=\beta b[a_1,a_2]+(a+\gamma b)[a_1,a_3]\\
 &=\beta(a+\gamma b)a_2+
 \bigl(\gamma a+(\beta+\gamma^2)b\bigr)a_3=d(a_3).
\end{aligned}
\]
For the third defining product,
\[
\begin{aligned}
[d(a_1),a_3]+[a_1,d(a_3)]
 &=\beta\bigl(\gamma a+(\beta+\gamma^2)b\bigr)a_2\\
 &\quad+\Bigl(\beta(a+\gamma b)+
 \gamma\bigl(\gamma a+(\beta+\gamma^2)b\bigr)\Bigr)a_3,
\end{aligned}
\]
while
\[
\begin{aligned}
 d(\beta a_2+\gamma a_3)
 &=\beta d(a_2)+\gamma d(a_3)\\
 &=\beta\bigl(\gamma a+(\beta+\gamma^2)b\bigr)a_2\\
 &\quad+\Bigl(\beta(a+\gamma b)+
 \gamma\bigl(\gamma a+(\beta+\gamma^2)b\bigr)\Bigr)a_3.
\end{aligned}
\]
Thus the three nonzero products are preserved.  All products with left factor $a_2$ or $a_3$ are zero, and the derivation identity for them is also zero on both sides.

For the second family, $\operatorname{char}\F=2$ and $\gamma=0$.  We have
\[
\begin{aligned}
 d(a_1)&=s a_1+a a_2+b a_3,\\
 d(a_2)&=\beta b a_2+a a_3,\\
 d(a_3)&=\beta a a_2+(s+\beta b)a_3.
\end{aligned}
\]
Now
\[
\begin{aligned}
[d(a_1),a_1]+[a_1,d(a_1)]
 &=s a_2+s a_2+a a_3+\beta b a_2\\
 &=\beta b a_2+a a_3=d(a_2),
\end{aligned}
\]
where $2s=0$.  Also,
\[
\begin{aligned}
[d(a_1),a_2]+[a_1,d(a_2)]
 &=s a_3+\beta b a_3+\beta a a_2\\
 &=\beta a a_2+(s+\beta b)a_3=d(a_3).
\end{aligned}
\]
Finally, using $[a_1,a_3]=\beta a_2$,
\[
\begin{aligned}
[d(a_1),a_3]+[a_1,d(a_3)]
 &=s\beta a_2+\beta a a_3+(s+\beta b)\beta a_2\\
 &=\beta^2 b a_2+\beta a a_3
 =\beta d(a_2)
 =d([a_1,a_3]).
\end{aligned}
\]
Again all products with left factor $a_2$ or $a_3$ give zero on both sides.  Thus both displayed families are exact.  Finally, direct multiplication of the matrices of the basis derivations shows that the first algebra is abelian, while in the exceptional case $[u,v]=v$ and $w$ is central.  This proves the structural assertions.
\end{proof}

\begin{remark}\label{rem:L16char2}
The second part of Theorem~\ref{thm:L16} is a genuine characteristic-two phenomenon.  When $\operatorname{char}\F=2$ and $\gamma=0$, the irreducible polynomial is $X^2-\beta$.  Such a polynomial may be irreducible over an imperfect field, and the coefficient $s$ in $d(a_1)$ is then free.  This case is not covered by characteristic-$\ne2$ calculations.
\end{remark}

\begin{remark}\label{rem:L16cyclic}
The algebra $L_{16}(\beta,\gamma)$ is one-generated, since $a_2=[a_1,a_1]$ and $a_3=[a_1,a_2]$.  The general one-generated result of~\cite{DerCyclic1} gives a two-dimensional abelian ideal in its derivation algebra and shows that an additional derivation can occur only when the characteristic divides $2$.  Theorem~\ref{thm:L16} makes this completely explicit in dimension three: the extra parameter occurs exactly when $\operatorname{char}\F=2$ and $\gamma=0$.
\end{remark}

\section{Summary}

For convenience, Table~\ref{tab:summary} collects all sixteen derivation algebras in the fixed basis.  The structural notation $u,v,w,z,t$ is the notation introduced with the corresponding result above; in particular, these are explicitly defined basis derivations, not additional abstract symbols.  The algebra parameters and the field restrictions are those in Table~\ref{tab:status}.  For each fixed algebra, the remaining matrix parameters are arbitrary elements of $\F$; any further case distinctions are indicated in the first column.

\begingroup
\scriptsize
\setlength{\tabcolsep}{2pt}
\renewcommand{\arraystretch}{1.38}
\begin{longtable}{@{}P{0.135\textwidth}P{0.645\textwidth}P{0.055\textwidth}P{0.13\textwidth}@{}}
\caption{Derivation algebras of the sixteen three-dimensional types.}\label{tab:summary}\\
\toprule
Type & Matrix form and internal structure of $D=\Der(L)$ & $\dim$ & Source\\
\midrule
\endfirsthead
\toprule
Type & Matrix form and internal structure of $D=\Der(L)$ & $\dim$ & Source\\
\midrule
\endhead
\midrule
\multicolumn{4}{r}{\emph{Continued on next page}}\\
\endfoot
\bottomrule
\endlastfoot
$L_1$ & $\left\{\begin{pmatrix}a&0&0\\b&d&0\\c&e&2a\end{pmatrix}\right\}$\par\smallskip $D=(\F v\oplus\F w\oplus\F t)\rtimes(\F u\oplus\F z)$; $\F u\oplus\F z$ is abelian and $[v,t]=-w$. & $5$ & \cite{DerLowDimsUMJ}\\
$L_2$ & $\left\{\begin{pmatrix}a&0&0\\b&a+b&0\\c&d&2a+b\end{pmatrix}\right\}$\par\smallskip $D=(\F w\oplus\F z)\rtimes(\F u\oplus\F v)$; both displayed summands are abelian. & $4$ & \cite{DerNilpotent}\\
$L_3(\alpha)$ & $\left\{\begin{pmatrix}a&0&0\\b&a+(1+\alpha)b&0\\c&d&2a+(1+\alpha)b\end{pmatrix}\right\}$\par\smallskip $D=(\F w\oplus\F z)\rtimes(\F u\oplus\F v)$; both displayed summands are abelian. & $4$ & Thm.~\ref{thm:L3}\\
$L_4$, $\operatorname{char}\F\ne2$ & $\left\{\begin{pmatrix}a&0&0\\b&a&0\\c&d&2a\end{pmatrix}\right\}$\par\smallskip $D=(\F w\oplus\F z)\rtimes(\F u\oplus\F v)$; both displayed summands are abelian. & $4$ & Thm.~\ref{thm:L4}\\
$L_4$, $\operatorname{char}\F=2$ & $\left\{\begin{pmatrix}a&p&0\\b&a+p&0\\c&d&0\end{pmatrix}\right\}$\par\smallskip $D=(\F z\oplus\F t)\rtimes H$, $H=\F u\oplus\F v\oplus\F w$, and $H=(\F(u+w)\rtimes\F v)\oplus\F u$. & $5$ & Thm.~\ref{thm:L4}\\
$L_5(\beta)$, $\operatorname{char}\F\ne2$ & $\left\{\begin{pmatrix}a&-\beta b&0\\b&a&0\\c&d&2a\end{pmatrix}\right\}$\par\smallskip $D=(\F w\oplus\F z)\rtimes(\F u\oplus\F v)$; both displayed summands are abelian. & $4$ & Thm.~\ref{thm:L5}\\
$L_5(\beta)$, $\operatorname{char}\F=2$ & $\left\{\begin{pmatrix}a&\beta b&0\\b&e&0\\c&d&0\end{pmatrix}\right\}$\par\smallskip $D=(\F w\oplus\F z)\rtimes H$, $H=\F u\oplus\F v\oplus\F t$, and $H=(\F v\rtimes\F u)\oplus\F(u+t)$. & $5$ & Thm.~\ref{thm:L5}\\
$L_6(\delta)$ & $\left\{\begin{pmatrix}a&-\delta b&0\\b&a+b&0\\c&d&2a+b\end{pmatrix}\right\}$\par\smallskip $D=(\F w\oplus\F z)\rtimes(\F u\oplus\F v)$; both displayed summands are abelian. & $4$ & Thm.~\ref{thm:L6}\\
$L_7$ & $\left\{\begin{pmatrix}0&0&0\\a&b&0\\c&0&0\end{pmatrix}\right\}$\par\smallskip $D=(\F u\rtimes\F v)\oplus\F w$, $[v,u]=u$, and $\F w\le\zeta(D)$. & $3$ & \cite{DerNonNilpotent3}\\
$L_8(\lambda)$, $\operatorname{char}\F=2$ & $\left\{\begin{pmatrix}0&0&0\\a&b&0\\c&a&0\end{pmatrix}\right\}$\par\smallskip $D=(\F u\rtimes\F v)\oplus\F w$, $[v,u]=u$, and $\F w\le\zeta(D)$. & $3$ & \cite{DerNonNilpotent3}\\
$L_9$ & $\left\{\begin{pmatrix}0&0&0\\a&b&0\\c&0&c\end{pmatrix}\right\}$\par\smallskip $D=(\F u\rtimes\F v)\oplus\F w$, $[v,u]=u$, and $\F w\le\zeta(D)$. & $3$ & Thm.~\ref{thm:L9}\\
$L_{10}(r)$ & $\left\{\begin{pmatrix}0&0&0\\a&b&0\\c&0&rc\end{pmatrix}\right\}$\par\smallskip $D=(\F u\rtimes\F v)\oplus\F w$, $[v,u]=u$, and $\F w\le\zeta(D)$. & $3$ & Thm.~\ref{thm:L10}\\
$L_{11}$ & $\left\{\begin{pmatrix}0&0&0\\a&b&0\\0&-a&2b\end{pmatrix}\right\}$\par\smallskip $D=\F u\rtimes\F v$, $[v,u]=u$. & $2$ & Thm.~\ref{thm:L11}\\
$L_{12}$ & $\left\{\begin{pmatrix}a&0&0\\b&2a&0\\c&b&3a\end{pmatrix}\right\}$\par\smallskip $D=(\F v\oplus\F w)\rtimes\F u$, with $[u,v]=v$ and $[u,w]=2w$. & $3$ & \cite{DerCyclic1,DerCyclic2}\\
$L_{13}$ & $\left\{\begin{pmatrix}0&0&0\\a&a&0\\b&a&0\end{pmatrix}\right\}$\par\smallskip $D=\F u\oplus\F v$ is abelian. & $2$ & \cite{DerNonNilpotent1}\\
$L_{14}(\rho)$, $\rho\ne1$ & $\left\{\begin{pmatrix}0&0&0\\a&a&0\\0&0&b\end{pmatrix}\right\}$\par\smallskip $D=\F u\oplus\F v$ is abelian. & $2$ & Thm.~\ref{thm:L14}\\
$L_{14}(1)$ & $\left\{\begin{pmatrix}0&0&0\\a&a&c\\b&b&d\end{pmatrix}\right\}$\par\smallskip $D\cong\mathfrak{gl}_2(\F)$. & $4$ & \cite{DerNonNilpotent1}\\
$L_{15}$, $\operatorname{char}\F\ne2$ & $\left\{\begin{pmatrix}0&0&0\\a&a&0\\b&a+b&a\end{pmatrix}\right\}$\par\smallskip $D=\F u\oplus\F v$ is abelian. & $2$ & \cite{DerNonNilpotent1}\\
$L_{15}$, $\operatorname{char}\F=2$ & $\left\{\begin{pmatrix}s&0&0\\a&a&s\\b&a+b&s+a\end{pmatrix}\right\}$\par\smallskip $D=(\F v\rtimes\F u)\oplus\F(v+w)$, $[u,v]=v$, and $\F(v+w)\le\zeta(D)$. & $3$ & \cite{DerNonNilpotent1}\\
$L_{16}(\beta,\gamma)$, $\operatorname{char}\F\ne2$ or $\gamma\ne0$ & $\left\{\begin{pmatrix}0&0&0\\a&\beta b&\beta(a+\gamma b)\\b&a+\gamma b&\gamma a+(\beta+\gamma^2)b\end{pmatrix}\right\}$\par\smallskip $D=\F u\oplus\F v$ is abelian. & $2$ & Thm.~\ref{thm:L16}\\
$L_{16}(\beta,0)$, $\operatorname{char}\F=2$ & $\left\{\begin{pmatrix}s&0&0\\a&\beta b&\beta a\\b&a&s+\beta b\end{pmatrix}\right\}$\par\smallskip $D=(\F v\rtimes\F u)\oplus\F w$, $[u,v]=v$, and $\F w\le\zeta(D)$. & $3$ & Thm.~\ref{thm:L16}\\
\end{longtable}
\endgroup

The table shows four characteristic-two increases in dimension: for $L_4$, $L_5(\beta)$, $L_{15}$, and the subfamily $L_{16}(\beta,0)$.  These cases are the main reason that a characteristic-$\ne2$ description cannot simply be transferred to arbitrary fields.

\sloppy

\fussy

\end{document}